\documentclass[12pt,reqno]{amsart}
\usepackage{amsmath}
\usepackage{amssymb,latexsym}
\usepackage[matrix,arrow]{xy}
\usepackage[usenames]{color}
\usepackage{slashed}
\usepackage{parskip}
\newtheorem{thm}{Theorem}[section]
\newtheorem{cor}[thm]{Corollary}

\newtheorem{prop}[thm]{Proposition}

\theoremstyle{definition}

\numberwithin{equation}{section}
 \theoremstyle{remark}

\newcommand{\bfh}{\textbf{h}}
\newcommand{\pab}{P'_{a,b}}
\newcommand{\n}{\nabla}

\newcommand{\ld}{\Lambda}
\newcommand{\intm}{\int_{M^4}}

\newcommand{\fp}{\Bbb {F}P^m}

\newcommand{\al}{\alpha}

\begin{document}
\title{\text The sharp estimates for the first eigenvalue of Paneitz operator on $4$-dimensional submanifolds}
\author []{ Daguang Chen  and Haizhong Li}
\email { dgchen@math.tsinghua.edu.cn,hli@math.tsinghua.edu.cn}
\address{Department of
Mathematical Sciences, Tsinghua University, Beijing, 100084, P.R.
China}

\thanks {The second author was partly supported by NSFC grant No.10971110. }

\begin{abstract}In this note,  we obtain the sharp estimates for the first
eigenvalue of Paneitz operator for $4$-dimensional compact submanifolds in Euclidean
space. Since unit spheres and  projective spaces can be canonically imbedded into
Euclidean space, the corresponding estimates for the first eigenvalue are also obtained.
\end{abstract}
\maketitle
\renewcommand{\sectionmark}[1]{}

\section{introduction}

Given a smooth 4-dimensional Riemannian manifold $(M^4,g)$, the Paneitz operator, discovered
in \cite{Paneitz83}, is the fourth-order operator defined by
\begin{equation*}
Pf=\Delta^2f-\text{div}\Big(\frac23 R \ I-2Ric\Big)df, \qquad \text{for} \quad f \in
C^\infty(M^4),
\end{equation*}
where $\Delta$ is the scalar Laplacian defined by $\Delta=\text{div} d$, $\text{div}$ is
the divergence with respect to $g$, $R, Ric$ are the scalar curvature and Ricci curvature
respectively. Beautiful works on Paneitz operator $P$ have been developed recently by
Chang \cite{Chang05}, Branson-Chang-Yang \cite{BCY92}, Chang-Yang \cite{CY95},
Chang-Gursky-Yang \cite{CGY99}, and Gursky \cite{Gursky98}.  The Paneitz operator was
generalized to higher dimensions by Branson \cite{Branson}.

In \cite{CHY2004,Gursky99,HR2001,XuYang2001}, the authors investigated positivity of the
Paneitz operator. In analogy with the conformal volume in \cite{LY82}, Xu and Yang
\cite{XuYang2002} defined n-conformal energy for compact 4-dimensional Riemannian
manifold immersed in an n-dimensional sphere $\Bbb S^n(1)$. In the same paper
\cite{XuYang2002}, the upper bound for the first eigenvalue of Paneitz operator was
bounded by using n-conformal energy. In \cite{ChenLi2009}, we obtained the sharp
estimates for the second eigenvalue of Paneitz operator for compact n($n\geq
7$)-dimensional submanifolds in an Euclidean space and an unit sphere.

Assume that $M^n$ is a compact Riemannian manifold immersed into Euclidean space $\Bbb
R^N$. In \cite{Reilly77}, Reilly  obtained the well-known estimate for the  first
eigenvalue $\lambda_1$ of Laplacian
\begin{equation}\label{Reilly77}
\lambda_1\leq \frac{n}{Vol(M)}\intm |\textbf{H}|^2dv,
\end{equation}
where $\textbf{H}$ is the mean curvature of immersion $M^n$ in $\Bbb R^N$. In
\cite{SI92}, Soufi and Ilias obtained the corresponding estimates for submanifolds in
sphere $\Bbb S^N(1)$, hyperbolic space $\Bbb H^N(-1)$ and some other ambient spaces.
Motivated by \cite{LY82}, Soufi and Ilias \cite{SI00} also obtained the sharp estimates
for the second eigenvalue of Schr\"{o}dinger operator for submanifolds in space form
$\Bbb R^N$, $\Bbb S^N(1)$ and hyperbolic space $\Bbb H^N(-1)$.

The aim of this paper is to obtain the Reilly-type sharp estimates for the first eigenvalue
of Paneitz operator in terms of the extrinsic geometry of the compact
4-dimensional submanifold $M^n$ in space form $R^N$ and Euclidean unit sphere $\Bbb S^{N-1}$.

\section{Preliminaries}
\subsection{Some formulaes in submanifolds}
Let $M$ be an n-dimensional compact ~Riemannian manifold and let
$X=(x^1,\cdots,x^N):M\hookrightarrow\mathbb{ R}^N$ be an isometric immersion of $M$ in
$\mathbb {R}^N$. For a fixed point $q\in M$, in a neighborhood of $q$, we can choose an
unit orthonomal basis $\{e_i,e_\alpha\}$ with $e_i$ tangent to $M$ and $e_\alpha$ normal
to $M$ for $1\leq i\leq n, n+1\leq \alpha \leq N$.

It is well known that (\cite{CC2008,CL2007,Li2002})

\begin{equation}\label{dr}
\sum_{ p=1}^N|\n x^p|^2=n,
\end{equation}
\begin{equation}\label{meancur}
\Delta X=\sum h_{ii}^\al e_\al=n\textbf{H},
\end{equation}
\begin{equation}\label{zero}
\sum_{p=1}^N\Delta x^p\n x^p=0,
\end{equation}
where $h^\alpha_{ij}$ is components of the second fundamental form, where
$\textbf{H}=\frac 1n\sum_{\alpha=n+1}^N\sum_{i=1}^n h^\al_{ii}e_\al$ is the mean
curvature vector field of $M$, $\Delta$ is defined by $ \Delta=\text{div}d$. Since the
ambient space is Euclidean space, the Gauss equation imply
\begin{equation}\label{Gauss}
R=n^2|\textbf{H}|^2-|\textbf{h}|^2
\end{equation}
where $\textbf{h}$, $R$ denote the second fundamental form and
scalar curvature of $M$, respectively.

\subsection{High order mean curvature}
Let $(M^n, g)$ be an n-dimensional submanifold in an $N$-dimensional space form
$R^N(\kappa)$,
 where $R^N(\kappa)$ is Euclidean space $\Bbb R^N$ when $\kappa=0$, $R^N(1)$ is
 a unit sphere $\Bbb S^N$ when $\kappa=1$.
 Suppose $\{e_i\}$ is a local orthonormal basis for $\Gamma(TM^n)$
with dual basis $\{\theta_i\}$ and $\{e_\alpha\}$ is a local orthonormal basis for the
normal bundle of $x : M^n \longrightarrow R^N(\kappa)$, denote
$\textbf{h}_{ij}=\sum_{\alpha=n+1}^Nh^\alpha_{ij}e_{\alpha}$, $(\textbf{h}_{ij})$ is the
vector matrix with respect to the frame $\{e_1, \dots,  e_n\}$ on $M^n$. We define the
following $(0, 2)$-tensor $T_r$ for $r \in\{0,1, \dots,  n-1\}$, (see
\cite{CL2007,Grosjean02,Reilly77}): If $r$ is even, we set
\begin{equation}\label{Trenven}
\begin{aligned}
T_r&=\frac{1}{r!} \sum_{\substack{i_1\cdots i_r i\\j_1\cdots j_r j}}\delta_{j_1\cdots j_r j}^{i_1\cdots i_r
i}<\bfh_{i_1j_1},\bfh_{i_2j_2}>\cdots<\bfh_{i_{r-1}j_{r-1}},\bfh_{i_rj_r}>\theta_i\otimes\theta_j\\
&=\sum_{i,j}T^i_{rj}\theta_i\otimes\theta_j,
\end{aligned}
\end{equation}
where $\delta_{j_1\cdots j_r j}^{i_1\cdots i_r i}$ are the
generalized Kronecker symbols and
\begin{equation}\label{Trij}
T^i_{rj}=\frac{1}{r!} \sum_{\substack{i_1\cdots i_r\\j_1\cdots j_r}}\delta_{j_1\cdots j_r j}^{i_1\cdots i_r
i}<\bfh_{i_1j_1},\bfh_{i_2j_2}>\cdots<\bfh_{i_{r-1}j_{r-1}},\bfh_{i_rj_r}>.
\end{equation}
We also set
\begin{equation}\label{Tr-1}
\begin{aligned}
T_{r-1}&=\frac{1}{(r-1)!} \delta_{j_1\cdots j_r j}^{i_1\cdots i_r
i}<\bfh_{i_1j_1},\bfh_{i_2j_2}>\cdots<\bfh_{i_{r-3}j_{r-3}},\bfh_{i_{r-2}j_{r-2}}>\bfh_{i_{r-1}j_{r-1}}
\theta_i\otimes\theta_j\\
&=T^\alpha_{r-1 ij}\theta_i\otimes\theta_je_\alpha,
\end{aligned}
\end{equation}
where
\begin{equation}\label{Talpha}
T^\alpha_{r-1 ij}=\frac{1}{(r-1)!} \delta_{j_1\cdots j_r
j}^{i_1\cdots i_r
i}<\bfh_{i_1j_1},\bfh_{i_2j_2}>\cdots<\bfh_{i_{r-3}j_{r-3}},\bfh_{i_{r-2}j_{r-2}}>h^\alpha_{i_{r-1}j_{r-1}}.
\end{equation}

For any even integer $r\in\{0,1,\cdots, n-1\}$, the $r-$th mean curvature function $H_r$
and $(r+1)$-th mean curvature vector field $\textbf{H}_{r+1}$ as follows
\begin{equation}\label{hmean}
H_r=\displaystyle{\frac{1}{\left(
               \begin{array}{c}
                 n \\
                 r \\
               \end{array}
             \right)}}\frac 1rT^\alpha_{r-1 ij}h^\alpha_{ij},
\end{equation}
\begin{equation}\label{sr+1}
\textbf{H}_{r+1}=\displaystyle{\frac1{\left(
               \begin{array}{c}
                 n \\
                 r+1 \\
               \end{array}
             \right)}}\frac{1}{r+1}T^i_{rj}h^\alpha_{ij}e_\alpha£¬
\end{equation}
where $\left(
         \begin{array}{c}
           n \\
           r \\
         \end{array}
       \right)
$ is  the binomial coefficient. By convention, we put
$H_0=1.$

\subsection{ Submanifolds in projective space}
Let $\Bbb F$ denote the field $\Bbb R$ of real numbers, the field
$\Bbb C$ of complex numbers or the field $\Bbb Q$ of quaternions.
Let $\Bbb FP^m$ denote the $m$-dimensional projective space over
$\Bbb F$. The projective space $\Bbb FP^m$ are endowed  standard
Riemannian metric whose  sectional curvature  are either constant
and equal to $1$ or pinched between $1$ and $4$ $( \Bbb F=\Bbb C \
\mbox{or}\ \Bbb Q)$.
 It is well known that there exist the first standard embedding of
 projective spaces into Euclidean spaces (see \cite{Chen84}).
 Let $\rho: \fp\longrightarrow \textsl{H}_{m+1}(\Bbb F)$
 be the first standard embedding of projective spaces into Euclidean
 spaces, where
 $$\textsl{H}_{m+1}(\Bbb F)=\left\{A\in  \textsl{M}_{m+1}(\Bbb F)|A=\bar
 A^t\right\},
 $$
$\textsl{M}_{m+1}(\Bbb F)$ denote the space of $(m+1)\times (m+1)$
matrices over $\Bbb F$. For convenience, we introduce the integers
$$
d(\Bbb F)=dim_{\Bbb R}\Bbb F=\left \{\begin{array}{ll} 1  &\qquad
\mbox{if}\quad\Bbb F=\Bbb R\\
2  &\qquad \mbox{if}\quad \Bbb F=\Bbb C\\
4  &\qquad \mbox{if}\quad\Bbb F=\Bbb Q.
\end{array}
\right.
$$

\begin{prop} (see Lemma 6.3 in  \cite{Chen84})
Let $f: M^n\longrightarrow \fp$ be an isometric immersion and let $\textbf{H}$ and
$\textbf{H}^\prime$ be the mean curvature vector fields of the immersion $f$ and $\rho
\circ f$ respectively. Then one obtains
$$
|H^\prime|^2=|H|^2+ \frac{4(n+2)}{3n}+\frac2{3n^2}\sum_{i\neq
j}\tilde K(e_i, e_j)
$$
where  $\{e_i\}^n_{i=1}$ is the local  orthonormal basis of $\Gamma(TM)$, $\tilde K$ is
the section curvature of $\fp$ which can be expressed by
$$
\tilde K(e_i, e_j)=\left \{
\begin{array}{ll}
1, & \hbox{if}\ \Bbb F=\Bbb R ; \\
1+3(e_i\cdot Je_j)^2, & J\ \mbox{is the complex structure of}\ \Bbb{C}P^m, \hbox{if}\ \Bbb F=\Bbb C;\\
1+3\sum_{r=1}^3(e_i\cdot J_r e_j)^2, & J_r\ \mbox {is the
quaternionic structure of}\ \Bbb {Q}P^m, \hbox{if}\ \Bbb F=\Bbb Q .
\end{array}
\right.
$$
\end{prop}

In fact, one can infer that
$$
|\textbf{H}^\prime|^2=\left \{
\begin{array}{ll}
|\textbf{H}|^2+\frac{2(n+1)}{n}, &\mbox{for}\ \Bbb {R}P^m\\
|\textbf{H}|^2+\frac{2(n+1)}{n}+\frac2{n^2}\sum_{i,j=1}^n(e_i\cdot Je_j)^2\leq
|\textbf{H}|^2+\frac{2(n+2)}{n}, &\mbox{for}\ \Bbb {C}P^m\\
|\textbf{H}|^2+\frac{2(n+1)}{n}+\frac2{n^2}\sum_{i,j=1}^n\sum_{r=1}^3(e_i\cdot
J_re_j)^2\leq |\textbf{H}|^2+\frac{2(n+4)}{n}, &\mbox{for}\ \Bbb {Q}P^m
\end{array}
\right.
$$
i.e.
\begin{equation}\label{meanineq}
|\textbf{H}^\prime|^2\leq |\textbf{H}|^2+\frac {2(n+d(\Bbb F))}{n}
\end{equation}
where the equality holds if and only if $M$ is a complex submanifold
of $\Bbb {C}P^m$, for the case $\Bbb {C}P^m$; the equality holds if
and only if $n\equiv 0 \ (\mbox{mod 4})$ and $M$ is a quarternionic
submanifold of $\Bbb {Q}P^m$, for the case $\Bbb {Q}P^m$.

\section{Compact submanifolds in Euclidean space}

Assume that $M^4$ is a 4-dimensional compact  Riemannian manifold and isometrically
immerse into Euclidean space $\Bbb R^N$. Let $X=(x^1,\dots,x^N):M^4\longrightarrow \Bbb
R^N$ be the position vector of $M^4$ in $\Bbb R^N$. The eigenvalue problem for Paneitz
operator refer to
\begin{equation}\label{eigenPaneit}
Pf=\ld f,\qquad f\in C^\infty(M^4).
\end{equation}
For the first eigenvalue of Paneitz operator, we have
\begin{thm}\label{Euclidean}
Assume that $M^4$ is a 4-dimensional compact  Riemannian manifold
isometrically immersed in Euclidean space $\Bbb R^N$. Let $\ld_1$ be
the first nonzero eigenvalue of Paneitz operator, then
\begin{equation}\label{Reilly1}
\ld_1 \leq \frac{\Big(16\int_{M^4}|\textbf{H}|^2dv+\frac23\int_M
Rdv\Big)\int_{M^4}|\textbf{H}|^2dv}{\displaystyle{Vol}^2(M^4)},
\end{equation}
where $\text{Vol}(M^4)$ is the volume of manifold $M^4$. Moreover, if $N=5$, then the
equality holds in (\ref{Reilly1}) if and only if $M^4$ is a hypershere in $\Bbb R^5$; if
$N\geq 6$, and the equality holds in (\ref{Reilly1}), then $M^4$ is a minimal submanifold
of some hypersphere in $\Bbb R^N$.
\end{thm}
\begin{proof}Since all the quantities appearing in (\ref{Reilly1}) is independent of the choice of origin,
we may, without loss of generality,  assume that the center of gravity of $X$ is located at the origin; that is,
\begin{equation*}
    \int_{M^4}Xdv=0.
\end{equation*}
By min-max principle, using the coordinates functions $x^p \ (1\leq
p\leq N)$ as test functions, we obtain
\begin{equation} \label{V}
\ld_1\int_{M^4}|X|^2dv\leq \int_{M^4}\langle X,PX\rangle dv
\end{equation}
where $\langle,\rangle$ is Euclidean inner product.

Now we will calculate the term $\int_{M^4}\langle X,PX\rangle dv$ in
the following,
\begin{equation*}
\begin{aligned}
\int_{M^4}\langle X,PX\rangle dv&=\int_{M^4}\langle X,\Delta^2X\rangle
dv+\intm\sum_{p=1}^N \Big(\big(2R_{ij}-\frac23
R\delta_{ij}\big)x_i^p\Big)_{,j}x^p dv\\
&=\intm|\Delta X|^2dv-\intm \sum_{p=1}^N \Big(2R_{ij}-\frac23
R\delta_{ij}\Big)\langle x_i^p,x^p_j\rangle dv\\
&=\intm|\Delta X|^2dv-\intm \Big(2R_{ij}-\frac23 R\delta_{ij}\Big)\delta_{ij} dv\\
&=16\intm|\textbf{H}|^2dv+\frac23\intm R dv.
\end{aligned}
\end{equation*}
Therefore, we obtain
\begin{equation}\label{E1}
\ld_1\intm|X|^2dv\leq16\intm|\textbf{H}|^2dv+\frac23\intm Rdv.
\end{equation}
Multiply $\intm|\textbf{H}|^2$ in (\ref{E1}) both sides, then
\begin{equation}\label{E2}
\ld_1\intm|X|^2dv\intm|\textbf{H}|^2dv\leq\left(16\intm|\textbf{H}|^2dv+\frac23\intm
Rdv\right)\intm|H|^2dv.
\end{equation}
By Cauchy-Schwartz inequality, we obtain
\begin{equation}\label{E3}
    \left(\intm\langle X,\textbf{H}\rangle dv\right)^2\leq \intm|X|^2dv\intm|\textbf{H}|^2dv.
\end{equation}
Recall the Minkowski formula (see \cite{Reilly77})
\begin{equation}\label{E4}
    -\intm\langle X,\textbf{H}\rangle dv=\text{Vol}(M^4).
\end{equation}
From (\ref{E2}),(\ref{E3}) and (\ref{E4}), we get the inequality
(\ref{Reilly1}).

If $N\geq 6$ and the equality holds in (\ref{Reilly1}), all inequalities become qualities
from (\ref{V}) to (\ref{E4}). The equality holds in (\ref{E3}), then
\begin{equation}\label{mp}
    \textbf{H}=cX,\qquad  c\text{\ is\ constant}.
\end{equation}
The Minkowski formula (\ref{E4}) implies the constant $c\neq 0$.
Since
\begin{equation*}
d\langle X,X\rangle=2\langle d X,X\rangle=0,
\end{equation*}
it follows that
\begin{equation*}
|X|=r=constant.
\end{equation*}
Hence  $M^4$ is a   immersed submanifold into  a hypersphere $S^{N-1}(r)$ in  $\Bbb R^N$.
Assume that $\bar {h}(u,v)$ is the second fundamental form of $\Bbb S^{N-1}(r)$ in $\Bbb
R^N$, where $u, v$ are the tangent vectors of $M^4$. We have the following the relation
(See (3.5) of P.79 in \cite{Chen73})
$$
\textbf{H}=\tilde {\textbf{H}}+\frac14 \sum_{i=1}^4\bar h(e_i,e_i)
$$
where $\{e_i\}_{i=1}^4$ is the orthonormal basis in $\Gamma(TM^4)$,  $\tilde
{\textbf{H}}$ is the mean curvature vector of  $M^4$ in $S^{N-1}(r)$. Since  $\bar
h(e_i,e_j)=-\frac {1}{r^2}\delta_{ij} X$, we have
\begin{equation}\label{hh}
\textbf{H}=\tilde {\textbf{H}}-\frac {1}{r^2}X.
\end{equation}
From (\ref{mp}), we have $\tilde {\textbf{H}}=0$, i.e., $M^4$ is a  minimally immersed
submanifold in a hypersphere $S^{N-1}(r)$.

For $N=5$, that is, $M^4$ is a hypersurface in $\Bbb R^5$, the proof follows in much the
same way, we omit it here.
\end{proof}

\begin{cor}
Assume that $M^4$ is a compact 4-dimensional Riemannian manifold
isometrically immersed in Euclidean space $\Bbb R^N$. Let $\ld_1$ be
the first nonzero eigenvalue of Paneitz operator, then
\begin{equation}\label{Reilly2}
\ld_1\leq 24 \left(\frac{\intm|\textbf{H}|^2dv}{\text{Vol}(M^4)}\right)^2.
\end{equation}
In particular, equality holds in (\ref{Reilly2}) if and only if
$M^4$ is a round sphere.
\end{cor}
\begin{proof}
From Gauss equation, we get
\begin{equation*}
    R=16|\textbf{H}|^2-|\textbf{h}|^2.
\end{equation*}
By Cauchy-Schwartz inequality $|\textbf{h}|^2\geq 4|\textbf{H}|^2$ , we obtain
\begin{equation}\label{E5}
R\leq 12|\textbf{H}|^2.
\end{equation}
Combining  (\ref{Reilly1}) with (\ref{E5}), we deduce the estimate
(\ref{Reilly2}).

If the equality holds in (\ref{Reilly2}), then $|\textbf{h}|^2= 4|\textbf{H}|^2$, so we
have $M^4$ is a round sphere.
\end{proof}

\begin{cor}
Assume that $M^4$ is a compact 4-dimensional Riemannian manifold
isometrically immersed in Euclidean space $\Bbb R^N$. Let $\ld_1$ be
the first nonzero eigenvalue of Paneitz operator, then
\begin{equation}\label{Reilly3}
\ld_1\leq24\frac{\intm|\textbf{H}|^4dv}{\text{Vol}(M^4)}.
\end{equation}
The equality holds if and only if $M^4$ is a round sphere.
\end{cor}
\begin{proof}
From Cauchy-Schwartz inequality, we get
\begin{equation*}
    \left(\intm|\textbf{H}|^2dv\right)^2=\left(\intm|\textbf{H}|^2\cdot1dv\right)^2\leq\text{Vol}(M^4)\intm|\textbf{H}|^4dv.
\end{equation*}
Then by inequality (\ref{Reilly2}), (\ref{Reilly3}) holds.
\end{proof}

\begin{cor}\label {corReilly}
Assume that $M^4$ is a compact 4-dimensional Riemannian manifold isometrically immersed
in Euclidean space $\Bbb R^N$. Let $\ld_1$ be the first nonzero eigenvalue of Paneitz
operator, then
\begin{enumerate}
 \item If $N\geq 6$, then
 \begin{equation}\label{Reilly4}
 \ld_1 \left(\intm H_2dv\right)^2\leq \left(16\intm |\textbf{H}|^2dv+\frac23\intm
 Rdv\right)\intm|\textbf{H}_3|^2dv,
 \end{equation}
If  equality holds in (\ref{Reilly4}) and $\textbf{H}_3$ does not vanish identically,
 then $M^4$ is a immersed submanifold with $\textbf{H}_3=cX$ in some hypersphere in $\Bbb
 R^N$ for some constant $c$.
\item If $N=5$, then
\begin{equation}\label{Reilly5}
\ld_1 \left(\intm H_{i-1}dv\right)^2\leq \left(16\intm
|H|^2dv+\frac23\intm
 Rdv\right)\intm|H_i|^2dv,\quad i=1,2,3,4.
\end{equation}
Equality holds in (\ref{Reilly5}) if and only if $M^4$ is a round sphere in $\Bbb R^5$.
\end{enumerate}

\end{cor}

\begin{proof} (1) \qquad Multiplying  $\intm|\textbf{H}_3|^2$ both sides in (\ref{E1}) , we
get
\begin{equation}\label{E6}
\ld_1\intm |X|^2dv\intm|\textbf{H}_3|^2dv\leq \left(16\intm |\textbf{H}|^2dv+\frac23\intm
 Rdv\right)\intm|\textbf{H}_3|^2dv.
\end{equation}
From Cauchy-Schwartz inequality
\begin{equation*}
\intm\langle X,\textbf{H}_3\rangle^2dv\leq \intm |X|^2dv\intm|\textbf{H}_3|^2dv
\end{equation*}
and Minkowski formula in \cite{Reilly77}
\begin{equation*}
-\intm\langle X,\textbf{H}_3\rangle dv=\intm H_2dv,
\end{equation*}
the inequality (\ref{Reilly4}) holds.

 From the basic facts for Cauchy-Schwarz inequality,
this implies that $\textbf{H}_3=c X$ for some constant $c$.
 By hypothesis, $\textbf{H}_3$ does not vanish identically, we have $c\neq 0$.
 Since $d|X|^2=2\langle X, dX\rangle=0$, it follows that $|X|=const.$,
 so $X$ map $M^4$ in to  a hypersphere of $\Bbb R^{N}$. \\

(2) \qquad Multiplying $\int H_i^2$ both sides in (\ref{E1}), then we have
\begin{equation}\label{E7}
\ld_1\intm |X|^2H_i^2dv\leq \left(16\intm |\textbf{H}|^2dv+\frac23\intm
 Rdv\right)\intm H_i^2dv.
\end{equation}
From Cauchy-Schwartz inequality
\begin{equation*}
\intm\langle X,H_i e_5\rangle^2dv\leq \intm |X|^2H_i^2dv
\end{equation*}
and Minkowski formula in \cite{Reilly77}
\begin{equation*}
-\intm\langle X,H_i e_5\rangle dv=\intm H_{i-1} dv,
\end{equation*}
where $e_5$ is the unit normal vector field of $M^4$, the inequality (\ref{Reilly4})
holds. If the equality holds in (\ref{Reilly5}),  we have $H_ie_5=cX$ from the
Cauchy-Schwarz inequality for some constant $c$. Since
\begin{equation*}
d\langle X,X\rangle=2\langle d X,X\rangle=0,
\end{equation*}
it follows that $|X|=r=const.$. Therefore, we have $M^4$ is a sphere in $\Bbb R^5$.
\end{proof}

\section{submanifolds in spheres and projective spaces}
Since the unit sphere and projective spaces admit canonical
embedding into the Euclidean space, we will obtain the corresponding
estimates for the first eigenvalue of  Paneitz operator.

\begin{thm}\label{sphere}
Assume that $M^4$ is a 4-dimension compact Riemannian manifold and
isometrically immerses into the unit sphere $\Bbb S^{N-1}\subset
\Bbb R^{N-1}$. Let $\ld_1$ be the first nonzero eigenvalue of
Paneitz operator, then
\begin{equation}\label{Reilly6}
\begin{aligned}
 \ld_1 &\leq 16+\frac{32\intm|\textbf{H}|^2dv+\frac23\intm
R dv}{\text{Vol}(M^4)}\\
&\qquad+\frac{1}{\text{Vol}^2(M^4)}\left(16\intm|\textbf{H}|^2dv+\frac23\intm R
dv\right)\intm|\textbf{H}|^2dv.
\end{aligned}
\end{equation}
Moreover, if the equality holds in (\ref{Reilly6}), then $M^4$ is a minimal submanifold
in unit sphere $\Bbb S^{N-1}$.

\end{thm}
\begin{proof}
Since the unit sphere can be canonically imbedded into Euclidean
space, we have the following diagram
\begin{equation}\label{diagram1}
    M^4\xymatrix{
  \ar[dr]_{\textbf{H}'} \ar[r]^{\textbf{H}}
                & \Bbb S^{N-1} \ar[d]^{i}  \\
                & \Bbb R^N             }
\end{equation}
where $H'$ is the mean curvature vector of $M^4$ in Euclidean space $\Bbb
R^N$.

 By the theorem \ref{Euclidean} and $|\textbf{H}'|^2=|\textbf{H}|^2+1$, we obtain
\begin{equation*}
\begin{aligned}
\ld_1 \leq &\frac{\Big(16\intm(|\textbf{H}|^2+1)dv+\frac23\intm
R dv\Big)\Big(\intm(|\textbf{H}|^2+1)dv\Big)}{\text{Vol}^2(M^4)}\\
=&\frac{\left(16\text{Vol}(M^4)+16\intm|\textbf{H}|^2dv+\frac23\intm
Rdv\right)\left(\intm|\textbf{H}|^2dv+\text{Vol}(M^4)\right)}{\text{Vol}^2(M^4)}\\
=&16\left(1+\frac{\intm|\textbf{H}|^2dv}{\text{Vol}(M^4)}\right)+\frac{16\intm|\textbf{H}|^2dv+\frac23\intm
Rdv}{\text{Vol}(M^4)}\\
& \ \quad+\frac{\left(16\intm |\textbf{H}|^2dv+\frac23\intm
Rdv\right)\intm|\textbf{H}|^2dv}{\text{Vol}^2(M^4)}.
\end{aligned}
\end{equation*}
Therefore we get (\ref{Reilly6}).

It is well known that the mean curvatures $\textbf{H}'$ and $\textbf{H}$ obey
\begin{equation}\label{mmcurvature}
\textbf{H}'=\textbf{H}-X,
\end{equation}
where $X$ is the position vector field of $M^4$ in unit sphere $\Bbb S^{N-1}$. From
Theorem \ref{Euclidean}, if the equality holds in (\ref{Reilly6}), then
\begin{equation*}
    \textbf{H}'=cX, \qquad \text{ for some constant c.}
\end{equation*}
From (\ref{mmcurvature}), if the equality holds in (\ref{Reilly6}), then $M^4$ is a
minimal submanifold in $\Bbb S^{N-1}$.
\end{proof}

From (\ref{Reilly2}), (\ref{Reilly3}) and
\begin{equation*}
\begin{aligned}
  R&=12+16|\textbf{H}|^2-|\textbf{h}|^2\\
  &\leq12(|\textbf{H}|^2+1),
  \end{aligned}
\end{equation*}
we get from Theorem \ref{sphere}
\begin{cor}
Assume that $M^4$ is a 4-dimension compact Riemannian manifold and
isometrically immerses into the unit sphere $\Bbb S^{N-1}\subset
\Bbb R^{N-1}$. Let $\ld_1$ be the first nonzero eigenvalue of
Paneitz operator, then
\begin{equation}\label{Reilly8}
\ld_1\leq \frac{24}{\text{Vol}^2(M^4)} \left(\intm \Big(|\textbf{H}|^2+1\Big)dv\right)^2
\end{equation}
and
\begin{equation}\label{Reilly9}
\ld_1 \leq \frac{24}{\text{Vol}(M^4)} \intm \left(|\textbf{H}|^2+1\right)^2dv.
\end{equation}
Moreover, the equalities hold in (\ref{Reilly8}) and (\ref{Reilly9}) if and only if $M^4$
is a totally geodesic submanifold in unit sphere $\Bbb S^{N-1}$.
\end{cor}

Since there is a canonical imbedding from $\Bbb FP^m(\Bbb F=\Bbb
R,\Bbb C, \Bbb Q)$ to Euclidean space $H_{m+1}(\Bbb F)$, then for
compact manifold $M^4$ isometrically immersed into the projective
space $\Bbb FP^m$, we have the following diagram
\begin{equation}\label{projectivespace}
 M^4\xymatrix{
  \ar[dr]_{\textbf{H}'} \ar[r]^{\textbf{H}}
                & \Bbb FP^m \ar[d]^{\rho}  \\
                & H_{m+1}{(\Bbb F)}              }
\end{equation}
where $\textbf{H}'$ is the mean curvature vector field in Euclidean space $H_{m+1}(\Bbb
F)$. From the relation (\ref{meanineq}) and Theorem 3.1, it follows
\begin{thm}Assume that $M^4$ is a compact Riemannian manifold
isometrically immersed into the projective space $\Bbb FP^m$, Let
$\ld_1$ be the first nonzero eigenvalue of Paneitz operator, then
\begin{equation}\label{Reilly 10}
\begin{aligned}
\ld_1 &\leq 4\left(4+d(\Bbb F)\right)^2+\frac{1}{\text{Vol}(M^4)}\left(4+d(\Bbb
F)\right)\left(16\intm|\textbf{H}|^2dv+\frac13\intm
Rdv\right)\\&\qquad+\frac{\left(16\intm|\textbf{H}|^2dv+\frac23\intm
Rdv\right)\left(\intm |\textbf{H}|^2dv\right)}{\text{Vol}^2(M^4)}.
\end{aligned}
\end{equation}
\end{thm}

\section{Eigenvalue of Paneitz-like operator}
In this section, we will give the Reilly-type inequalities for Paneitz-like operator for
$n$-dimensional ($n\geq 4$) compact submanifolds in Euclidean space.

The eigenvalue problem for Paneitz-like operator refer to
\begin{equation}\label{eigenPaneit}
\pab f=\Delta^2f-{\rm div}(aR\ I +b Ric)df=\ld f,\qquad f\in C^\infty(M^n)
\end{equation}
where constants $a$ and $b$ satisfies $na+b\geq 0$. When $n=4$, $a=\frac 23$ and $b=-2$,
$P'_{a,b}$ is the Paneitz operator.

 Taking the similar argument as Paneitz operator, the theorems in section 3 and
section 4 can be generalized:
\begin{thm}\label{Euclidean1}
Assume that $M^n$ is a compact n-dimensional Riemannian manifold isometrically immersed
in Euclidean space $\Bbb R^N$. Let $\ld_1$ be the first nonzero eigenvalue of
Paneitz-like operator $\pab$, then
\begin{equation}\label{Reilly11}
\ld_1 \leq \frac{\Big((n^2\intm |\textbf{H}|^2dv+(na+b)\intm
Rdv\Big)\Big(\int_{M^n}|\textbf{H}|^2dv\Big)}{\displaystyle{Vol}^2(M^n)},
\end{equation}
where $\text{Vol}(M^n)$ is the volume of manifold $M^n$. Moreover, the equality holds in
(\ref{Reilly11}),  $M^n$ is
 a minimal submanifold
of some hypersphere in $\Bbb R^N$.
\end{thm}
\begin{cor}
Assume that $M^n$ is a compact n-dimensional Riemannian manifold isometrically immersed
in Euclidean space $\Bbb R^N$. Let $\ld_1$ be the first nonzero eigenvalue of
Paneitz-like operator, then
\begin{equation}\label{Reilly12}
\ld_1\leq  n\left[n+(n-1)(na+b)\right] \frac{\intm|\textbf{H}|^4dv}{\text{Vol}(M^n)}.
\end{equation}
The equality holds if and only if $M^n$ is a round sphere in  $\Bbb R^N$ .
\end{cor}
\begin{cor}
Assume that $M^n$ is a compact n-dimensional Riemannian manifold isometrically immersed
in Euclidean space $\Bbb R^N$. Let $\ld_1$ be the first nonzero eigenvalue of
Paneitz-like operator, then
\begin{enumerate}
 \item If $N\geq n+2$ and r is an even integer, $0\leq r\leq n-1$, then
 \begin{equation}\label{Reilly13}
 \ld_1 \left(\intm H_{r}dv\right)^2\leq \left(n^2\intm |\textbf{H}|^2dv+(na+b)\intm
Rdv\right)\intm|\textbf{H}_{r+1}|^2dv;
 \end{equation}
\item If $N=n+1$ and $r$ is any integer, $0\leq r\leq n$,  then
\begin{equation}\label{Reilly14}
\ld_1 \left(\intm H_{r-1}dv\right)^2\leq \left(n^2\intm
|H|^2dv+(na+b)\intm Rdv\right)\intm|H_r|^2dv.
\end{equation}
\end{enumerate}
Equalities hold in  case (2), $M^n$ is a sphere in
$\Bbb R^{n+1}$.
\end{cor}

\begin{thm}
Assume that $M^n$ is a n-dimension compact Riemannian manifold and isometrically immerses
into the unit sphere $\Bbb S^{N-1}\subset \Bbb R^{N}$. Let $\ld_1$ be the first nonzero
eigenvalue of Paneitz operator, then
\begin{equation}\label{Reilly15}
\begin{aligned}
 \ld_1 &\leq n^2+\frac{2n^2\intm |\textbf{H}|^2dv+(na+b)\intm
Rdv}{\text{Vol}(M^n)}\\
&\qquad+\frac{n^2(\intm |\textbf{H}|^2dv)^2+(na+b)\intm|\textbf{H}|^2dv\intm
Rdv}{\text{Vol}^2(M^n)}.
\end{aligned}
\end{equation}
\end{thm}

\begin{cor}
Assume that $M^n$ is a n-dimension compact Riemannian manifold and isometrically immerses
into the unit sphere $\Bbb S^{N-1}\subset \Bbb R^{N}$. Let $\ld_1$ be the first nonzero
eigenvalue of Paneitz-like operator, then
\begin{equation}\label{Reilly17}
\ld_1\leq
n\left[n+(n-1)(na+b)\right]\left(\frac{\intm|\textbf{H}|^2dv}{\text{Vol}(M^n)}+1\right)^2.
\end{equation}
\end{cor}

\begin{thm}Assume that $M^n$ is a compact Riemannian manifold
isometrically immersed into the projective space $\Bbb FP^m$, Let $\ld_1$ be the first
nonzero eigenvalue of Paneitz-like operator, then
\begin{equation}\label{Reilly 19}
\begin{aligned}
\ld_1 \leq& 4n^2\left(1+\frac{d(\Bbb
F)}{n}\right)+2\left(1+\frac{d(\Bbb
F)}{n}\right)\left(2n^2+na+b\right)\frac{\intm Rdv}{\text{Vol}(M^n)}
\\&+\frac{n^2\intm
|\textbf{H}|^2dv+(na+b)\intm|\textbf{H}|^2dv\intm Rdv}{\text{Vol}^2(M^n)}.
\end{aligned}
\end{equation}
\end{thm}

\vspace{1cm}

 \textbf{Acknowledgements}\qquad

The authors would like to express their thanks to the referee for some useful comments.


\providecommand{\bysame}{\leavevmode\hbox
to3em{\hrulefill}\thinspace}

\end{document}